\documentclass[reqno,11pt]{amsart}
 \usepackage[dvips]{graphicx}
 \usepackage{psfrag}
 \usepackage{amsfonts, amssymb, amsmath, amsthm}
 \usepackage{color}
\newtheorem{thm}{Theorem}[section]
\newtheorem{prop}[thm]{Proposition}
\newtheorem{lem}[thm]{Lemma}

\def\Holder{{H\"{o}lder}}

\def\bdy #1{\partial #1}

\def\bbR{{\mathbb R}}

\def\div{{\operatorname{div}}}

\def\rI{{\rm I}}

\def\p{{\partial}}

\def\Forall{\forall\hspace{2pt}}

\def\({{(\hspace{-2pt}(}}
\def\){{)\hspace{-2pt})}}

\def\XXint#1#2#3{{\setbox0=\hbox{$#1{#2#3}{\int}$}
\vcenter{\hbox{$#2#3$}}\kern-.5\wd0}}

\begin{document}
\baselineskip = 18pt

\title[2D GBBM equation]{Well-posedness of the two-dimensional generalized Benjamin-Bona-Mahony equation
on the upper half plane}

\author[Y.-C. Lin, C. H. A. Cheng, J. M. Hong]{Ying-Chieh Lin, C. H. Arthur Cheng, John M. Hong}
\address{Ying-Chieh Lin, C. H. Arthur Cheng, and John M. Hong \endgraf
         Department of Mathematics \endgraf
         National Central University \endgraf
         Chung-Li, Taiwan 32001, ROC}
\email{linyj@math.ncu.edu.tw; cchsiao@math.ncu.edu.tw; jhong@math.ncu.edu.tw}

%

\author[J. Wu]{Jiahong Wu}
\address{Jiahong Wu \endgraf
         Department of Mathematics \endgraf
         Oklahoma State University \endgraf
         401 Mathematical Sciences \endgraf
         Sillwater, OK 74078, USA}
\email{jiahong@math.okstate.edu}

\author[J.-M. Yuan]{Juan-Ming Yuan}
\address{Juan-Ming Yuan \endgraf
         Department of Financial and Computational Mathematics \endgraf
         Providence University \endgraf
         Taichung, Taiwan 43301, ROC}
\email{jmyuan@pu.edu.tw}

\thanks{}



\date{}
\maketitle


\begin{abstract}
This paper focuses on the two-dimensional Benjamin-Bona-Mahony and
Benjamin-Bona-Mahony-Burgers equations with a general flux function.
The aim is at the global (in time)
well-posedness of the initial-and boundary-value problem for these equations
defined in the upper half-plane. Under suitable growth conditions on the
flux function, we are able to establish the global well-posedness in a Sobolev
class. When the initial- and boundary-data become more regular,
the corresponding solutions are shown to be classical. In addition, the
continuous dependence on the data is also obtained.
\end{abstract}


\section{Introduction}
\setcounter{equation}{0}

This paper is concerned with the two-dimensional (2D) Benjamin-Bona-Mahony-Burgers
equation of the form
\begin{subequations}\label{GBBMB_eq}
\begin{alignat}{2}
u_t + \div\left(\phi(u)\right) &= \nu_1 \Delta u + \nu_2 \Delta u_t \qquad&&\text{in}\quad \Omega \times (0, T)\,,\\
u &= g &&\text{on}\quad \Omega \times \{t=0\} \,,\\
u &= h &&\text{on}\quad \bdy \Omega \times (0,T)\,,
\end{alignat}
\end{subequations}
where $\Omega\subseteq \bbR^2$ denotes a smooth domain, $u=u(x,t)$ is
a scalar function,
$\phi(u)$ is a vector-valued flux function, and $\nu_1\ge 0$ and $\nu_2>0$ are real
parameters. (\ref{GBBMB_eq}a) is a natural generalization of the 1D
Benjamin-Bona-Mahony equation
\begin{equation}\label{BBM}
u_t + u_x + uu_x - u_{xxt}=0,
\end{equation}
which governs the unidirectional propagation of 1D long waves with small amplitudes. Therefore (\ref{GBBMB_eq}a) with $\nu_1=0$ is sometimes called the generalized Benjamin-Bona-Mahony (GBBM) equation while (\ref{GBBMB_eq}a) with $\nu_1>0$ the Benjamin-Bona-Mahony-Burgers (GBBM-B) equation. (\ref{GBBMB_eq}a) is also a regularization of the scalar conservation law
\begin{equation}
u_t + \div (\phi(u)) = 0\,. \label{conservation_law}
\end{equation}
In addition, (\ref{GBBMB_eq}a) has also been derived to model the two-phase fluid flow in a porous medium, as in the oil recovery. In fact, (\ref{GBBMB_eq}a) is a special case of the well-known Buckley-Leverett equation
\begin{equation}\label{BL}
u_t + \div(\phi(u)) = -\div\{H(u)\nabla(J(u)-\tau u_t)\},
\end{equation}
where $u$ denotes the saturation of water,
the functions $\phi$, $H$ and $J$ are related
to the capillary pressure and the permeability of water and oil \cite{BBLL}.
(\ref{GBBMB_eq}a) follows from (\ref{BL}) by
linearizing the static capillary pressure $J(u)$ and $H$ around a constant state.

\smallskip
Attention here will be focused on the case when $\Omega=\bbR^2_+$,
the upper half-plane. The aim is at the global well-posedness of \eqref{GBBMB_eq}
with inhomogeneous boundary data, namely $h\not \equiv 0$. One motivation behind
this study is to rigorously validate the laboratory
experiments involving water waves generated by a wavemaker
mounted at the end of a water channel. We are able to prove the global existence
and uniqueness of the mild and classical solutions to \eqref{GBBMB_eq}.
In addition, a continuous dependence result is also obtained. Our main theorems can
be stated as follows.

\begin{thm} [Existence and uniqueness] \label{Classical}
Let $\nu_1 \ge 0$, $\nu_2>0$, and $\Omega = \bbR^2_+$. Let $T>0$.
Suppose that $(g,h) \in H^2(\Bbb R^2_+) \times C^1([0,T];H^2(\Bbb R))$,
and the flux $\phi\in C^2(\Bbb R,\Bbb R^2)$ satisfies the conditions
\begin{equation}\label{phi-1}
\phi(0)=0\quad\text{and}\quad\|\phi''\|_{L^\infty(\bbR)} \le C.
\end{equation}
Then \eqref{GBBMB_eq} admits a unique mild solution $u\in C([0,T];H^2(\Bbb R^2_+))$.
If we further assume that $(g,h) \in C^{2,\alpha}_{\rm loc}(\Bbb R^2_+)\times C^1([0,T];C^{2,\alpha}_{\rm loc}(\Bbb R))$ for some $0<\alpha<1$,
then the mild solution is in fact a classical solution.
\end{thm}

We remark that the condition $\phi(0)=0$ in \eqref{phi-1} can be removed. In the case of $\phi(0)\ne 0$,
we define the new function $\widetilde{\phi}(s)=\phi(s)-\phi(0)$ for all $s\in\Bbb R$, then $\widetilde{\phi}(0)=0$ and
$\div(\widetilde{\phi}(v))=\div(\phi(v))$.
We can consider the new equation by replacing the function $\phi$ with $\widetilde{\phi}$ in (\ref{GBBMB_eq}a).

\begin{thm}[Continuous dependence on data]\label{stability}
Let $\nu_1 \ge 0$, $\nu_2>0$, and $\Omega = \bbR^2_+$. Suppose that $\phi\in C^2(\Bbb R,\Bbb R^2)$ satisfies the condition \eqref{phi-1}. Then the mild solution obtained in Theorem {\rm\ref{Classical}} depends continuously on the initial datum $g$
and boundary datum $h$.
If we further assume that $\phi\in C^3(\Bbb R,\Bbb R^2)$ satisfies
\begin{equation}\label{phi-3}
\|\phi'''\|_{L^\infty(\bbR)} \le C \,,
\end{equation}
then the same result also holds for the classical solution.
\end{thm}

It is worth remarking that Theorems \ref{Classical} and
\ref{stability} hold with either $\nu_1=0$ or with $\nu_1>0$ and
do not rely on the regularization due to the Burgers dissipation.
We briefly review related well-posedness results and then explain
the main difficulties in proving Theorems \ref{Classical} and
\ref{stability}. There is a very large literature on the global
well-posedness and asymptotic behavior of solutions for the 1D BBM
equation on the whole line (see, e.g. \cite{ABS,BBM,Bona_Chen,M}).
Extensive results have also been
obtained on the global well-posedness for the initial- and
boundary-value problem of the BBM equation posed on the half-line
(see, e.g. \cite{BB,BCSZ1,BCSZ2,BD,BL1,BL2,BPS,BT,BW,HWY2,MMM}).
In particular, in the well-known articles \cite{BB,BBM}, the
existence of classical solutions and their continuous dependence
on the specified data were investigated. While current analytic
results for the multi-dimensional BBM or BBM-Burgers equations
only dealt with the existence of mild solutions on either the
whole space or bounded domains with homogeneous boundary data
(see, e.g. \cite{AG,GF,GW,K}). The results presented here allow
inhomogeneous boundary data, which correspond to the setup of a
wavemaker mounted at the end of a channel in laboratory
experiments. We emphasize that, our methods are also suitable for
the corresponding initial value problem. Therefore, Theorems
\ref{Classical} and \ref{stability} are the complete extensions of
the results in \cite{BB,BBM} to the multi-dimensional case.

\smallskip
The main difficulty in proving Theorem \ref{Classical} is from two sources. First, the
Green function for operator $\rI-\Delta$ in 2D is much more singular than the 1D case; and second, the inhomogeneous boundary data prevents us from obtaining a time-independent $H^1$ upper bound, which very much simplifies the process of global-in-time estimates. To overcome the difficulties, we introduce a new function
that assumes the homogeneous boundary data and rewrite the equation in an integral
form through the Green function of the elliptic operator. In addition, we use the bootstrapping technique to obtain the classical solution of \eqref{GBBMB_eq} instead of looking for the solution in classical spaces directly.

\smallskip
The rest of this paper is divided into six sections. The first five sections
deal with the case when $\nu_1 =0$ while the last section explains why the
results for $\nu_1=0$ can be extended to the case when $\nu_1>0$. Section
\ref{sec:alternative_formulation} introduces a new function that assumes
homogenous boundary data and converts (\ref{GBBMB_eq}) into an integral
formulation in terms of this new function. Section \ref{sec:1-lap_est}
presents preliminary regularity estimates for the operator $(I-\Delta)^{-1}$ in
the Sobolev space $H^2$ and in H\"{o}lder spaces. Section \ref{sec:local-in-time}
proves that (\ref{GBBMB_eq}) has a unique local (in time) classical solution.
We make use of the integral representation \eqref{rewrite}.
Section \ref{sec:global-in-time} establishes the global existence and uniqueness
of the local solution obtained in Section \ref{sec:local-in-time} by showing
global bounds for the solution in $H^2$ and in H\"{o}lder spaces. Section \ref{sec:stability} contains the continuous dependence results. The continuous
dependence of the solution on the initial data and the boundary data is proven
in two functional settings and the proof is lengthy. As aforementioned,
Section \ref{sec:nu1ne0} is devoted to the case when $\nu_1>0$.

\vskip .3in
\section{An alternative formulation}\label{sec:alternative_formulation}
\setcounter{equation}{0}

In this section we set $\nu_1=0$. The case when $\nu_1>0$ is handled in
Section \ref{sec:nu1ne0}. This section provides an integral
formulation of (\ref{GBBMB_eq}).

\smallskip
In order to apply the standard elliptic theory in the functional framework of Sobolev spaces, we shall rewrite equation (\ref{GBBMB_eq}) with homogeneous boundary data. This is achieved by setting $v(x,t)=u(x,t)-h(x_1,t)e^{-x_2}$, which satisfies
\begin{subequations}\label{GBBM2}
\begin{alignat}{2}
(\rI - \Delta) v_t + \div \left(\phi(v+he^{-x_2})\right) &= h_{x_1 x_1 t} e^{-x_2} \qquad&&\text{in}\quad\bbR^2_+ \times (0,T), \\
v &= \widetilde{g} &&\text{on}\quad\bbR^2_+ \times \{t=0\}, \\
v &= 0 &&\text{on}\quad \bdy\bbR^2_+ \times (0,T).
\end{alignat}
\end{subequations}
where
\begin{equation}\label{defn:gtilde}
\widetilde{g}(x) = g(x) - h(x_1,0) e^{-x_2}\,.
\end{equation}
Denoting $(\rI - \Delta)^{-1} f$ as the unique solution to the elliptic equation
\begin{subequations}\label{Dirichlet}
\begin{alignat}{2}
(\rI-\Delta)u &=f \qquad &&\text{in}\quad \Omega,\\
u &= 0\qquad &&\text{on}\quad \partial\Omega,
\end{alignat}
\end{subequations}
we can formally write the solution $v$ of \eqref{GBBM2} via the integral representation
\begin{equation}\label{IEv}
\aligned
v(x,t)&=\widetilde{g}(x)+(\rI-\Delta)^{-1}\left(\{ h_{x_1 x_1}(x_1,t)-h_{x_1 x_1}(x_1,0)\}e^{-x_2}\right)\\
&\quad -\int^t_0(\rI-\Delta)^{-1}\left\{\div\left(\phi(v+he^{-x_2})\right)\right\}d\tau\,.
\endaligned
\end{equation}
For short, we rewrite \eqref{IEv} as the form
\begin{equation}\label{rewrite}
v=\mathcal{A}v=\widetilde{g}+\mathcal{B}h+\mathcal{C}v,
\end{equation}
where, for $x\in\Bbb R^2_+$ and $t\ge 0$,
\begin{equation*}
\aligned
\mathcal{B}h(x,t)&:=(\rI-\Delta)^{-1}\left(\{ h_{x_1 x_1}(x_1,t)-h_{x_1 x_1}(x_1,0)\}e^{-x_2}\right),\\
\mathcal{C}v(x,t)&:=-\int^t_0(\rI-\Delta)^{-1}\left\{\div\left(\phi(v+he^{-x_2})\right)\right\}d\tau.
\endaligned
\end{equation*}

\vskip .3in
\section{Preliminary results}
\label{sec:1-lap_est}

This section specifies the functional spaces and provides two preliminary estimates on
the solutions to the elliptic equation \eqref{Dirichlet}.  In the rest of this paper, we write $C^k([0,T];H^\ell(\bbR^2_+))$ for the space
$$
\Big\{u: [0,T] \to H^\ell(\bbR^2_+)\,\Big|\, \lim_{t\to t_0} \sum_{j=0}^k \Big\|\frac{\p^k u}{\p t^k}(t) - \frac{\p^k u}{\p t^k}(t_0)\Big\|_{H^\ell(\bbR^2_+)} = 0 \quad\Forall t_0\in [0,T] \Big\}
$$
equipped with norm
$$
\|u\|_{C^k_t H^\ell_x} := \max_{t\in [0,T]}\sum_{j=0}^k \Big\|\frac{\p^k u}{\p t^k}(t)\Big\|_{H^\ell(\bbR^2_+)}\,.
$$
The spaces with the particular indices  $k=0,1$ and $\ell=1,2$ will
be frequently used. For the simplicity of notation, when $k=0$, we omit the super-index $0$, that is, $C([0,T];H^\ell(\bbR^2_+)) \equiv C^0([0,T];H^\ell(\bbR^2_+))$ and $\|\cdot\|_{C_t H^\ell_x} \equiv \|\cdot\|_{C^0_t H^\ell_x}$. We will also need
the space
$$
C([0,T];L^p(\bbR^2_+)) \equiv \Big\{u:[0,T] \to L^p(\bbR^2_+)\,\Big|\, \lim_{t\to t_0} \|u(t) - u(t_0)\|_{L^p(\bbR^2_+)} =0 \Big\}
$$
equipped with norm
$$
\|u\|_{C_t L^p_x} = \max_{t\in [0,T]} \|u(t)\|_{L^p(\bbR^2_+)}\,.
$$
Similar notation is used to define the space of the boundary data which is only defined on the real line $\bbR$. We introduce
$$
C([0,T];H^2(\bbR)) \equiv \Big\{h:[0,T] \to H^2(\bbR)\,\Big|\, \lim_{t\to t_0} \|h(t) - h(t_0)\|_{H^2(\bbR)} =0 \Big\}
$$
equipped with norm
$$
\|h\|_{C_t H^2_{x_1}} = \max_{t\in [0,T]} \|h(t)\|_{H^2(\bbR)}\,.
$$

To study the classical solutions, we let $C^{k,\alpha}(\Omega)$ denote the space of $k$-times classically differentiable functions whose $k$-th derivatives are \Holder\ continuous with exponent $\alpha$. The norm on $C^{k,\alpha}(\Omega)$ is given by
$$
\|u\|_{C^{k,\alpha}(\Omega)} = \sum_{j=0}^k \sup_{x\in \Omega} |D^j u(x)| + \sup_{x,y\in \Omega} \frac{|D^k u(x) - D^k u(y)|}{|x-y|^\alpha}\,,
$$
where $D^j u$ denotes the $j$-th classical derivative of $u$.

\smallskip
To deal with the integral representation (\ref{IEv}), we need some crucial estimates on the operator $(\rI - \Delta)^{-1}$. In particular, the bounds in the following propositions will be employed in the subsequent sections.

\begin{prop}\label{reg-1}
Assume $f\in L^2(\Omega)$. Then the Dirichlet problem \eqref{Dirichlet}
admits a unique solution $u\in H^2(\Omega)$.
Furthermore,
$$
\|u\|_{H^2(\Omega)}\le C\| f\|_{L^2(\Omega)}\,.
$$
\end{prop}

If $f$ is instead in a H\"older space,
then we have the following H\"older's estimates for the solution of \eqref{Dirichlet}.

\begin{prop}\label{reg-2}
Assume that $\Omega\subset\Bbb R^2$ is a smooth domain. Assume that $f$
is in $C^{0,\alpha}_{\text{\rm loc}}(\Omega)\cap L^2(\Omega)$ for some $0<\alpha<1$. Then the solution $u$ of
the Dirichlet problem \eqref{Dirichlet} lies in $C^{2,\alpha}_{\text{\rm loc}}(\Omega)\cap H^2(\Omega)$. Furthermore,
for any compact subsets $\Omega_1$ and $\Omega_2$ of $\Omega$ with $\Omega_2\subset\subset\Omega_1$,
$$\| u\|_{C^{2,\alpha}(\Omega_2)}\le C(\| f\|_{C^{0,\alpha}(\Omega_1)}+\| f\|_{L^2(\Omega_1)}),$$
where $C>0$ depends only on the distance between $\Omega_2$ and $\partial\Omega_1$.
\end{prop}

\begin{proof}
By Proposition \ref{reg-1}, we have $u\in H^2(\Omega)$ and $\| u\|_{H^2(\Omega)}\le C\| f\|_{L^2(\Omega)}$.
Sobolev embedding theorem says that $u\in C^{0,\alpha}_{\rm loc}(\Omega)$ and
$\| u\|_{C^{0,\alpha}(\Omega')}\le C\| u\|_{H^2(\Omega')}$ for any compact subset $\Omega'$.
Thus, we get that $\Delta u=u-f\in C^{0,\alpha}_{\text{\rm loc}}(\Omega)$.
It follows from Lemma 4.2 and Theorem 4.6 in \cite{GT} that $u\in C^{2,\alpha}_{\text{\rm loc}}(\Omega)$
and, for any compact subsets $\Omega_1$ and $\Omega_2$ of $\Omega$ with $\Omega_2\subset\subset\Omega_1$, we have
$$\aligned
\| u\|_{C^{2,\alpha}(\Omega_2)}&\le C(\| f\|_{C^{0,\alpha}(\Omega_1)}+\| u\|_{C^{0,\alpha}(\Omega_1)}) \le C(\| f\|_{C^{0,\alpha}(\Omega_1)}+\| u\|_{H^2(\Omega_1)})\\
&\le C(\| f\|_{C^{0,\alpha}(\Omega_1)}+\| f\|_{L^2(\Omega_1)}),
\endaligned$$
where $C>0$ depends only on the distance between $\Omega_2$ and $\partial\Omega_1$.
\end{proof}

\vskip .3in

\section{Local-in-time existence}
\label{sec:local-in-time}
\setcounter{equation}{0}

This section proves that (\ref{GBBMB_eq}) has a unique local (in time) classical solution. We make use of the integral representation \eqref{rewrite}.
Due to the difficulty of applying the contraction mapping principle in the setting
of H\"{o}lder spaces, the proof is divided into two steps. The first step applies the
contraction mapping principle to \eqref{rewrite} in the setting of Sobolev spaces
to obtain a unique local solution. The second step obtains the desired regularity
of the local solution through a bootstrapping procedure.

\begin{lem}\label{H^2_S}
Let $(g,h) \in H^2(\Bbb R^2_+)\times C([0,T];H^2(\Bbb R))$, and $\phi$ satisfy the condition \eqref{phi-1}.
Then there is $S$ with $0<S\le T$, depending
only on $g$ and $h$, such that \eqref{rewrite} has a unique solution $v\in C([0,S];H^2(\Bbb R^2_+))$.
\end{lem}

\begin{proof}
This local existence and uniqueness result is proven through the contraction mapping
principle. More precisely, we show that $\mathcal{A}$ defined in \eqref{rewrite}
is a contraction map from $B(0,R)\subset C([0,S];H^2(\Bbb R^2_+))$ to itself, where
$B(0,R)$ denotes the closed ball centered at $0$ with radius $R$ in $C([0,S];H^2(\Bbb R^2_+))$. $S$ and $R$ will be specified later in the proof. It follows from \eqref{phi-1}, \eqref{rewrite}, Proposition \ref{reg-1} and the mean value theorem
that, for $v,w\in B(0,R)$,
\begin{equation}\label{estimate-Av-1}
\aligned
\|\mathcal{A}v\|_{C_t H^2_x}&\le \| g\|_{H^2}+C\| h\|_{C_t H^2_{x_1}}
+CS\|\div(\phi(v+he^{-x_2}))\|_{C_t L^2_x}\\
&\le \| g\|_{H^2}+C\| h\|_{C_t H^2_{x_1}}\\
&\quad +CS\|\phi'(v+he^{-x_2})\|_{C_t L^\infty_x}\|\nabla(v+he^{-x_2})\|_{C_t L^2_x}\\
&\le C_1+C_2 S(1+R)R
\endaligned
\end{equation}
and
\begin{equation}\label{estimate-Av-2}
\aligned
&\|\mathcal{A}v-\mathcal{A}w\|^2_{C_t H^2_x} \le CS^2\|\div(\phi(v+he^{-x_2}))-\div(\phi(w+he^{-x_2}))\|^2_{C_t L^2_x}\\
&\quad\le CS^2\sum^2_{j=1}\int_{\Bbb R^2_+}\left(\left|\phi'_j(v+he^{-x_2})(v-w)_{x_j}\right|^2\right.\\
&\qquad\qquad\quad +\left.\left|\left(\phi'_j(v+he^{-x_2})-\phi'_j(w+he^{-x_2})\right)(w+he^{-x_2})_{x_j}\right|^2\right)dx\\
&\quad\le CS^2\left\{\|\phi'(v+he^{-x_2})\|^2_{C_t L^\infty_x}\| v-w\|^2_{C_t H^1_x}\right.\\
&\qquad\qquad\quad +\left.\|\phi''(\overline{v}+he^{-x_2})\|^2_{C_t L^\infty_x}
 \|\nabla(w+he^{-x_2})\|^2_{C_t L^4_x}\| v-w\|^2_{C_t L^4_x}\right\}\\
&\quad\le C_2 S^2(1+R)^2 \| v-w\|^2_{C_t H^2_x},
\endaligned
\end{equation}
where $\overline{v}$ lies between the line segment joining $v$ and $w$, $C_1$ is a constant depending on
$\| g\|_{H^2}$ and $\| h\|_{C_t H^2_{x_1}}$, and $C_2$ is a constant depending on $\| h\|_{C_t H^2_{x_1}}$.
Note that \eqref{estimate-Av-2} implies $\mathcal{A}$ is a continuous map of $C([0,S];H^2(\Bbb R^2_+))$ to itself.
According to \eqref{estimate-Av-1}, $\mathcal{A}$ maps $B(0,R)$ onto itself if
\begin{equation}\label{R}
R\ge C_1+C_2 S(1+R)R.
\end{equation}
Hence, by \eqref{estimate-Av-2}, $\mathcal{A}$ is a contraction mapping of this ball if $C_2 S(1+R)<1$.
These conditions will be met if we take $R=2C_1$ and find a positive value $S>0$ small enough such that
\begin{equation}\label{contraction}
C_2 S(1+2C_1)\le\frac{1}{2}.
\end{equation}
Now, let
\begin{equation*}
v_0(x,t)=\widetilde{g}(x)+\mathcal{B}h(x,t)
\end{equation*}
and
\begin{equation*}
v_n(x,t)=\mathcal{A}v_{n-1}(x,t)=v_0(x,t)+\mathcal{C}v_{n-1}(x,t)\qquad\text{for}\ n\ge 1.
\end{equation*}
The contraction mapping principle gives that the sequence $v_n(x,t)$ converges in $C([0,S];H^2(\Bbb R^2_+))$
to the unique solution $v$ of \eqref{rewrite} in the ball $\| v\|_{C_t H^2_x}\le R$.
\end{proof}

\smallskip
If the initial data $g$  and the boundary data $h$ are also H\"{o}lder, then the
corresponding solution can also be shown to be H\"{o}lder. This is achieved through
the Sobolev embeddings and a bootstrapping procedure.

\begin{lem}\label{regularity}
Assume $g\in C^{2,\alpha}_{\text{\rm loc}}(\Bbb R^2_+)\cap H^2(\Bbb R^2_+)$ and $h\in C^1([0,T];C^{2,\alpha}_{\text{\rm loc}}(\Bbb R)\cap H^2(\Bbb R))$ for some $0<\alpha<1$. Let $\phi$ satisfy the condition \eqref{phi-1}. Then any solution $v\in C([0,T]; H^2(\Bbb R^2_+))$ of \eqref{rewrite} actually belongs to $C([0,T]; H^2(\Bbb R^2_+))\cap C^1([0,T]; C^{2,\alpha}_{\text{\rm loc}}(\Bbb R^2_+))$.
\end{lem}

\begin{proof}
Since $v\in C([0,T];H^2(\Bbb R^2_+))$ and $h\in C^1([0,T];H^2(\Bbb R))$, we have
\begin{equation*}
\div(\phi(v+he^{-x_2}))\in C([0,T]; H^1(\Bbb R^2_+)).
\end{equation*}
Proposition \ref{reg-1} gives
\begin{equation}\label{elliptic}
\aligned
(\rI-\Delta)^{-1}\{\div(\phi(v+he^{-x_2}))\}
&\in C([0,T]; H^3(\Bbb R^2_+))\\
&\hookrightarrow C([0,T];C^{1,\alpha}(\Bbb R^2_+)).
\endaligned
\end{equation}
On the other hand, $h\in C^1([0,T];C^{2,\alpha}_{\text{\rm loc}}(\Bbb R))$ implies
$$h_{x_1 x_1 t}e^{-x_2}\in C([0,T];C^{0,\alpha}_{\text{\rm loc}}(\Bbb R^2_+)).$$
Proposition \ref{reg-2} then yields
\begin{equation}\label{holder}
(\rI-\Delta)^{-1}\{ h_{x_1 x_1 t}e^{-x_2}\}\in C([0,T];C^{2,\alpha}_{\text{loc}}(\Bbb R^2_+)).
\end{equation}
In view of \eqref{elliptic} and \eqref{holder}, we obtain
\begin{equation*}
v_t=(\rI-\Delta)^{-1}\{ h_{x_1 x_1 t}e^{-x_2}-\div(\phi(v+he^{-x_2}))\}
\in C([0,T];C^{1,\alpha}_{\text{loc}}(\Bbb R^2_+)),
\end{equation*}
which implies that
\begin{equation}\label{v}
\aligned
v(x,t)&=\widetilde{g}(x)+\int^t_0 v_\tau(x,\tau) d\tau \\
&= g(x)-h(x_1,0)e^{-x_2}+\int^t_0 v_\tau(x,\tau) d\tau \in C^1([0,T];C^{1,\alpha}_{\text{loc}}(\Bbb R^2_+)).
\endaligned
\end{equation}
Using \eqref{v} and Proposition \ref{reg-2}, we have
\begin{equation*}
v_t=(\rI-\Delta)^{-1}\{ h_{x_1 x_1 t}e^{-x_2}-\div(\phi(v+he^{-x_2}))\}\in C([0,T];C^{2,\alpha}_{\text{loc}}(\Bbb R^2_+)),
\end{equation*}
and hence $v\in C^1([0,T];C^{2,\alpha}_{\text{loc}}(\Bbb R^2_+))$.
\end{proof}

\vskip .3in
\section{Global-in-time existence}
\label{sec:global-in-time}
\setcounter{equation}{0}

This section shows that the local (in time) solution obtained in the previous section
can be extended into a global one. This is achieved by establishing a global
bound for $\|v(t)\|_{H^2}$ under the condition that the flux $\phi$ obeys suitable
growth condition. We start with a global $H^1$-bound.

\begin{lem}\label{H^1 estimates}
Suppose $(g,h) \in H^2(\Bbb R^2_+)\times C^1([0,S];H^2(\Bbb R))$, and $\phi$ satisfy the condition \eqref{phi-1}.
Then the solution $v$ of \eqref{rewrite} obtained in Lemma {\rm\ref{H^2_S}} satisfies the estimates
\begin{equation}\label{estimate1}
\begin{array}{l}
\displaystyle{} \| v\|^2_{H^1}\le \| g\|^2_{H^1}+CS\| h\|^2_{C^1_t H^2_{x_1}}(1+\| h\|_{C^1_t H^2_{x_1}}) \vspace{.1cm} \\
\displaystyle{} \hspace{45pt} +C(1+\| h\|_{C^1_t H^2_{x_1}})\int^t_0\| v\|^2_{H^1}ds
\end{array}
\end{equation}
and
\begin{equation}\label{estimate2}
\| v\|_{H^1}\le \Big[\| g\|^2_{H^1}+CS\| h\|^2_{C^1_t H^2_{x_1}}(1+\| h\|_{C^1_t H^2_{x_1}})\Big]^{1/2}e^{CS(1+\| h\|_{C^1_t H^2_{x_1}})},
\end{equation}
where $C>0$ is a constant depending only on $\phi$.
\end{lem}

\begin{proof}
Multiplying (\ref{GBBM2}a) by $v$ and integrating over $\Bbb R^2_+$, we get
\begin{equation}\label{H^1 1}
\aligned
\frac{1}{2}\frac{d}{dt}\int_{\Bbb R^2_+}(v^2+|\nabla v|^2)\, dx
&=\int_{\Bbb R^2_+}\{ h_{x_1 x_1 t}e^{-x_2}-\div(\phi(v+he^{-x_2}))\} v\, dx\\
&=\int_{\Bbb R^2_+}h_{x_1 x_1 t}e^{-x_2} v\, dx-\int_{\Bbb R^2_+}\div(v\phi(v+he^{-x_2}))\, dx\\
&\quad +\int_{\Bbb R^2_+}\phi(v+he^{-x_2})\cdot\nabla v\, dx.
\endaligned
\end{equation}
Since $v=0$ on $\partial\Bbb R^2_+$,
\begin{equation}\label{H^1 2}
\int_{\Bbb R^2_+}\div(v\phi(v+he^{-x_2}))\, dx=\int_{\partial\Bbb R^2_+}v\phi(v+he^{-x_2})\cdot n\, dS=0.
\end{equation}
Now let $\Phi\in C^1(\Bbb R;\Bbb R^2)$ satisfy $\Phi'=\phi$ and $\Phi(0)=0$. Then
\begin{equation}\label{H^1 3}
\aligned
&\int_{\Bbb R^2_+}\phi(v+he^{-x_2})\cdot\nabla v\, dx\\
&\qquad =\int_{\Bbb R^2_+}\phi(v+he^{-x_2})\cdot\nabla(v+he^{-x_2})\, dx
 -\int_{\Bbb R^2_+}\phi(v+he^{-x_2})\cdot\nabla(he^{-x_2})\, dx\\
&\qquad =\int_{\Bbb R^2_+}\div(\Phi(v+he^{-x_2}))\, dx
 -\int_{\Bbb R^2_+}\phi(v+he^{-x_2})\cdot\nabla(he^{-x_2})\, dx\\
&\qquad =\int_{\partial\Bbb R^2_+}\Phi(h)\cdot n\, dS
 -\int_{\Bbb R^2_+}\phi(v+he^{-x_2})\cdot\nabla(he^{-x_2})\, dx.
\endaligned
\end{equation}
By the Sobolev embedding $H^2(\Bbb R)\hookrightarrow L^q(\Bbb R)$ for all $q\ge 2$,
$$h\in C^1([0,S];H^2(\Bbb R))\subset C^1([0,S];L^q(\Bbb R)).$$
Employing the mean value theorem together with \eqref{phi-1} and the properties of $\Phi$, we obtain
$$
|\Phi(h)-\Phi(0)|\le C(|h|^2+|h|^3).
$$
Hence
\begin{equation}\label{H^1 4}
\aligned
\bigg|\int_{\partial\Bbb R^2_+}\Phi(h)\cdot n\, dS\bigg|
&=\int_{\partial\Bbb R^2_+}|\Phi(h)-\Phi(0)|\, dS\\
&\le C(\| h\|^2_{C_t L^2_{x_1}}+\| h\|^3_{C_t L^3_{x_1}})\\
&\le C(\| h\|^2_{C^1_t H^2_{x_1}}+\| h\|^3_{C^1_t H^2_{x_1}}).
\endaligned
\end{equation}
Applying the mean value theorem and \eqref{phi-1} again, we have
\begin{equation*}
|\phi(v+he^{-x_2})|\le C(|v+he^{-x_2}|+|v+he^{-x_2}|^2)\,.
\end{equation*}
As a consequence,
\begin{equation}\label{H^1 5}
\aligned
&\bigg|\int_{\Bbb R^2_+}\phi(v+he^{-x_2})\cdot\nabla(he^{-x_2})\, dx\bigg|\\
&\qquad \le C(\| v+he^{-x_2}\|_{L^2}+\| v+he^{-x_2}\|^2_{L^4})\| h\|_{H^1_{x_1}}\\
&\qquad \le C(\| v\|_{L^2}+\| h\|_{L^2_{x_1}}+\| v\|^2_{L^4}+\| h\|^2_{L^4_{x_1}})\| h\|_{H^1_{x_1}}\\
&\qquad \le C(\| v\|_{H^1}+\| v\|^2_{H^1}+\| h\|_{C^1_t H^2_{x_1}}+\| h\|^2_{C^1_t H^2_{x_1}})\| h\|_{C^1_t H^2_{x_1}}.
\endaligned
\end{equation}
From \eqref{H^1 1}-\eqref{H^1 5}, we can conclude that
\begin{equation*}
\aligned
\frac{d}{dt}\| v\|^2_{H^1}
&\le 2\bigg|\int_{\Bbb R^2_+}h_{x_1 x_1 t}e^{-x_2}v\, dx\bigg|
 +C(\| v\|^2_{H^1}+\| h\|^2_{C^1_t H^2_{x_1}})\| h\|_{C^1_t H^2_{x_1}}\\
&\le C(\| v\|_{H^1}+\| v\|^2_{H^1}+\| h\|_{C^1_t H^2_{x_1}}+\| h\|^2_{C^1_t H^2_{x_1}})\| h\|_{C^1_t H^2_{x_1}}\\
&\le C(\| v\|^2_{H^1}+\| h\|^2_{C^1_t H^2_{x_1}})(1+\| h\|_{C^1_t H^2_{x_1}}),
\endaligned
\end{equation*}
that is,
\begin{equation*}
\| v\|^2_{H^1}\le \| g\|^2_{H^1}+CS\| h\|^2_{C^1_t H^2_{x_1}}(1+\| h\|_{C^1_t H^2_{x_1}})+C(1+\| h\|_{C^1_t H^2_{x_1}})\int^t_0\| v\|^2_{H^1}ds,
\end{equation*}
where $C$ depends only on $\phi$. Gronwall's inequality gives
$$
\| v\|_{H^1}\le \{\| g\|^2_{H^1}+CS\| h\|^2_{C^1_t H^2_{x_1}}(1+\| h\|_{C^1_t H^2_{x_1}})\}^{1/2}e^{CS(1+\| h\|_{C^1_t H^2_{x_1}})}. \vspace{-.7cm}
$$
\end{proof}

Now we derive the $H^2$-estimates based on the $H^1$-estimates we just obtained.
\begin{lem}\label{H^2 estimates}
Suppose $(g,h) \in H^2(\Bbb R^2_+)\times h\in C^1([0,S];H^2(\Bbb R))$, and $\phi$ satisfy the condition \eqref{phi-1}.
Then the solution $v$ of \eqref{rewrite} obtained in Lemma {\rm\ref{H^2_S}} satisfies the estimate
\begin{equation*}
\| v\|_{H^2}\le C(1+S)^{1/2}\exp\left\{ CS(1+S)^{1/2}e^{CS}\right\},
\end{equation*}
where $C>0$ is a constant depending only on $g$, $h$ and $\phi$.
\end{lem}

\begin{proof}
Multiplying (\ref{GBBM2}a) by $\Delta u$
and then integrating on $\Bbb R^2_+$, we have
\begin{equation}\label{H^2 1}
\aligned
&\frac{1}{2}\frac{d}{dt}\int_{\Bbb R^2_+}(|\nabla v|^2+|\Delta v|^2)dx\\
&\qquad =\int_{\Bbb R^2_+}h_{x_1 x_1 t}e^{-x_2}\Delta v\, dx
 -\int_{\Bbb R^2_+}\div(\phi(v+he^{-x_2}))\Delta v\, dx.
\endaligned
\end{equation}
By the mean value theorem and \eqref{phi-1},
\begin{equation}\label{H^2 3}
\aligned
|\div(\phi(v+he^{x_2}))|&\le |\phi'(v+he^{-x_2})|\, |\nabla(v+he^{-x_2})|\\
&\le C(1+|v+he^{-x_2}|)\, |\nabla(v+he^{-x_2})|.
\endaligned
\end{equation}
Thus, H\"older's inequality gives
\begin{equation*}
\aligned
\frac{d}{dt}\int_{\Bbb R^2_+}(|\nabla v|^2+|\Delta v|^2)dx
&\le 2\| h_{x_1 x_1 t}\|_{L^2_{x_1}}\|\Delta v\|_{L^2}+C\|\nabla( v+he^{-x_2})\|_{L^2}\|\Delta v\|_{L^2}\\
&\quad +C\|v+he^{-x_2}\|_{L^4}\|\nabla(v+he^{-x_2})\|_{L^4}\|\Delta v\|_{L^2}.
\endaligned
\end{equation*}
By the Sobolev embedding $H^1(\Bbb R^2_+)\hookrightarrow L^4(\Bbb R^2_+)$ and
Young's inequality,
\begin{equation*}
\aligned
\frac{d}{dt}(\|\nabla v\|^2_{L^2}+\|\Delta v\|^2_{L^2})
&\le C(\| h\|^2_{C^1_t H^2_{x_1}}+\| h\|^4_{C^1_t H^2_{x_1}})\\
&\quad +C(1+\| v\|_{C_t H^1_x}+\| h\|_{C^1_t H^2_{x_1}})\| v\|^2_{H^2},
\endaligned
\end{equation*}
where $C>0$ depends only on $\phi$; that is,
\begin{equation}\label{H^2 3}
\aligned
\|\nabla v\|^2_{L^2}+\|\Delta v\|^2_{L^2}
&\le C\{\| g\|^2_{H^2}+S(\| h\|^2_{C^1_t H^2_{x_1}}+\| h\|^4_{C^1_t H^2_{x_1}})\}\\
&\quad +C(1+\| v\|_{C_t H^1_x}+\| h\|_{C^1_t H^2_{x_1}})\int^t_0 \| v\|^2_{H^2} ds.
\endaligned
\end{equation}
Combining \eqref{estimate1} and \eqref{H^2 3}, we obtain
\begin{equation}\label{H^2 4}
\aligned
\| v\|^2_{H^2}&\le C\{\| g\|^2_{H^2}+S\| h\|^2_{C^1_t H^2_{x_1}}(1+\| h\|_{C^1_t H^2_{x_1}})^2\}\\
&\quad +C(1+\| v\|_{C_t H^1_x}+\| h\|_{C^1_t H^2_{x_1}})\int^t_0 \| v\|^2_{H^2} ds.
\endaligned
\end{equation}
Applying \eqref{estimate2} to \eqref{H^2 4}, we have
\begin{equation*}
\| v\|^2_{H^2}\le C(1+S)+C(1+S)^{1/2}e^{CS}\int^t_0 \| v\|^2_{H^2} ds,
\end{equation*}
where $C$ depends only on $g$, $h$, and $\phi$.
Therefore, by Gronwall's inequality,
$$
\| v\|_{H^2}\le C(1+S)^{1/2}\exp\left\{ CS(1+S)^{1/2}e^{CS}\right\}
$$
which concludes the proof of the lemma.
\end{proof}

\vskip .3in
\section{Continuous dependence of the solution on data}
\label{sec:stability}
\setcounter{equation}{0}

This section is devoted to proving Theorem
\ref{stability}. That is, we establish the desired continuous dependence. For the sake of clarity, we will divide the rest of this section into two subsections.
The first subsection proves the continuous dependence in the regularity setting of
$H^2$ while the second subsection focuses on the continuous dependence in the intersection space of $H^2$ and a H\"{o}lder class.  The precise statements
can be found in the lemmas below.

\subsection{Continuous dependence in $H^2$}

Let $\mathcal{L}_{\rm m}$ denote the mapping that takes the data $g$ and $h$
to the corresponding solutions of \eqref{GBBMB_eq}. By Theorem \ref{Classical} we have
\begin{equation*}
\mathcal{L}_{\rm m}:X_{\rm m}=H^2(\Bbb R^2_+)\times C^1([0,T];H^2(\Bbb R))\longrightarrow C([0,T];H^2(\Bbb R^2_+)).
\end{equation*}
Since $H^2(\Bbb R^2_+)$ and $C^1([0,T];H^2(\Bbb R))$ are Banach spaces,
the space $X_{\rm m}$ equipped with the usual product topology is also a Banach space.

\begin{lem}\label{s1}
Suppose that $\phi\in C^2(\Bbb R,\Bbb R^2)$ satisfies the condition \eqref{phi-1}. Then $\mathcal{L}_{\rm m}$ is continuous.
\end{lem}

\begin{proof}
Let $(g_i,h_i)\in X_{\rm m}$ and $u_i=\mathcal{L}_{\rm m}(g_i,h_i)$ be the mild solution of \eqref{GBBMB_eq}
corresponding to the initial data $g_i$ and the boundary data $h_i$, $i=1,2$.
Set $v_i=u_i-h_i e^{-x_2}$, $i=1,2$. Then $v_i$ satisfies the following initial-boundary
value problem:
\begin{equation*}
\begin{cases}
(v_i)_t-\Delta (v_i)_t-(h_i)_{x_1 x_1 t}e^{-x_2}+\div\left(\phi(v_i+h_i e^{-x_2})\right)=0,\\
v_i(x,0)=g_i(x)-h_i(x_1,0)e^{-x_2}:=\widetilde{g}_i(x),\\
v_i\big((x_1,0),t\big)=0.
\end{cases}
\end{equation*}
Define $w=v_1-v_2$. Then $w$ satisfies:
\begin{equation}\label{s1-1}
\begin{cases}
w_t-\Delta w_t-h_{x_1 x_1 t}e^{-x_2}+\div\left(\phi(v_1+h_1 e^{-x_2})\right)
-\div\left(\phi(v_2+h_2 e^{-x_2})\right)=0,\\
w(x,0)=\widetilde{g}(x),\qquad w\big((x_1,0),t\big)=0,
\end{cases}
\end{equation}
where $\widetilde{g}=\widetilde{g}_1-\widetilde{g}_2$ and $h=h_1-h_2$.
In addition, we derive that $w$ satisfies the following integral equation:
\begin{equation}\label{s1-2}
\aligned
w(x,t) &= \widetilde{g}(x) \hspace{-1pt}+\hspace{-1pt} (\rI-\Delta)^{-1}\left(\{ h_{x_1 x_1}(x_1,t) \hspace{-1pt}-\hspace{-1pt} h_{x_1 x_1}(x_1,0)\}e^{-x_2}\right)\\
&\quad - \int^t_0(\rI-\Delta)^{-1} \left\{\div(\phi(v_1 \hspace{-1pt}+\hspace{-1pt} h_1 e^{-x_2}))
\hspace{-1pt}-\hspace{-1pt}  \div(\phi(v_2 \hspace{-1pt}+\hspace{-1pt} h_2 e^{-x_2}))\right\}d\tau.
\endaligned
\end{equation}
Given $\varepsilon>0$.
Suppose that the distance between $(g_1,h_1)$ and $(g_2,h_2)$ in $X_{\rm m}$ is small enough such that
\begin{enumerate}
\item[(a)] $\|\widetilde{g}\|_{H^2}\le\varepsilon$, \qquad (b) $\| h\|_{C^1_t H^2_{x_1}}\le\varepsilon$.
\end{enumerate}
Taking $H^2$ norm on both sides of \eqref{s1-2} and using Proposition \ref{reg-1}, we derive
\begin{equation}\label{s1-3}
\aligned
\| w\|_{H^2}&\le \|\widetilde{g}\|_{H^2}+\| h\|_{C^1_t H^2_{x_1}}\\
&\quad +C\int^t_0\|\div(\phi(v_1+h_1 e^{-x_2}))
-\div(\phi(v_2+h_2 e^{-x_2}))\|_{L^2}\, d\tau.
\endaligned
\end{equation}
Since
\begin{align}
&\div(\phi(v_1+h_1 e^{-x_2}))-\div(\phi(v_2+h_2 e^{-x_2})) \nonumber\\
&\quad =\phi'(v_1+h_1 e^{-x_2})\cdot\nabla(v_1+h_1 e^{-x_2})-\phi'(v_2+h_2 e^{-x_2})\cdot\nabla(v_2+h_2 e^{-x_2}) \nonumber\\
&\quad =\{\phi'(v_1+h_1 e^{-x_2})-\phi'(v_2+h_2 e^{-x_2})\}\cdot\nabla(v_1+h_1 e^{-x_2}) \label{s1-3.5}\\
&\qquad +\phi'(v_2+h_2 e^{-x_2})\cdot\nabla(w+he^{-x_2}), \nonumber
\end{align}
the mean value theorem and condition \eqref{phi-1} yield
\begin{equation}\label{s1-4}
\aligned
&\left|\div(\phi(v_1+h_1 e^{-x_2}))-\div(\phi(v_2+h_2 e^{-x_2}))\right|\\
&\quad \le C\{|w \hspace{-1pt}+\hspace{-1pt} he^{-x_2}|\cdot |\nabla(v_1 \hspace{-1pt}+\hspace{-1pt} h_1 e^{-x_2})| \hspace{-1pt}
 +\hspace{-1pt} (1+|v_2 \hspace{-1pt}+\hspace{-1pt} h_2 e^{-x_2}|)\cdot |\nabla(w+he^{-x_2})|\}.
\endaligned
\end{equation}
Applying \eqref{s1-4} and H\"older's inequality to \eqref{s1-3}, we obtain
\begin{equation*}
\aligned
\| w\|_{H^2}&\le \|\widetilde{g}\|_{H^2}+\| h\|_{C^1_t H^2_{x_1}}
+C\int^t_0\| w+he^{-x_2}\|_{L^4}\|\nabla(v_1+h_1 e^{-x_2})\|_{L^4}d\tau\\
&\quad +C\int^t_0\|\nabla(w+he^{-x_2})\|_{L^2}d\tau+C\int^t_0\| v_2+h_2 e^{-x_2}\|_{L^4}\|\nabla(w+he^{-x_2})\|_{L^4}d\tau.
\endaligned
\end{equation*}
By Sobolev's inequality,
\begin{equation}\label{s1-5}
\aligned
\| w\|_{H^2}&\le \|\widetilde{g}\|_{H^2}+\| h\|_{C^1_t H^2_{x_1}}
+C\int^t_0\| w+he^{-x_2}\|_{H^2}d\tau\\
&\le \|\widetilde{g}\|_{H^2}+C(1+T)\| h\|_{C^1_t H^2_{x_1}}
+C\int^t_0\| w\|_{H^2}d\tau,
\endaligned
\end{equation}
where $C$ depends only on $v_1$, $v_2$, $h_1$, $h_2$, and $\phi$.
Then Gronwall's inequality gives
\begin{equation}\label{s1-6}
\| w\|_{C_t H^2_x}\le \{\|\widetilde{g}\|_{H^2}+\| h\|_{C^1_t H^2_{x_1}}\}e^{CT}.
\end{equation}
Note that
$$w=v_1-v_2=u_1-u_2-(h_1-h_2)e^{-x_2}=\mathcal{L}_{\rm m}(g_1,h_1)-\mathcal{L}_{\rm m}(g_2,h_2)-he^{-x_2}.$$
Therefore, by \eqref{s1-6},
$$\aligned
\| \mathcal{L}_{\rm m}(g_1,h_1)-\mathcal{L}_{\rm m}(g_2,h_2)\|_{C_t H^2_x}
&\le \| w\|_{C_t H^2_x}+\| h\|_{C_t H^2_{x_1}}\\
&\le \{\|\widetilde{g}\|_{H^2}+\| h\|_{C^1_t H^2_{x_1}}\}e^{CT} \le e^{CT}\varepsilon.
\endaligned \vspace{-0.7cm}
$$
\end{proof}

\subsection{Continuous dependence in the intersection of $H^2$ and a H\"{o}lder space}
This subsection proves the continuous dependence in the setting of
the intersection of $H^2$ and a H\"{o}ler space. First, we introduce the metrics on the spaces $C^{2,\alpha}_{\rm loc}(\Omega)$ and
$C^1([0,T];C^{2,\alpha}_{\rm loc}(\Omega))$, where $\Omega$ can be $\Bbb R^2_+$ or $\Bbb R$.
Let $\{\Omega_i\}^\infty_{i=1}$ be an increasing sequence of compact subsets of $\Omega$ satisfy
\begin{enumerate}
\item[(i)] $\Omega_i\subset\subset\Omega_{i+1}$ for all $i\in\Bbb N$,
\item[(ii)] $\displaystyle\bigcup^\infty_{i=1}\Omega_i=\Omega$.
\end{enumerate}
For a function $f\in C^1([0,T];C^{2,\alpha}_{\rm loc}(\Omega))$, we define
$$\rho_i(f)=\| f\|_{C^1([0,T];C^{2,\alpha}(\Omega_i))}\qquad\text{for}\ i\in\Bbb N.$$
Then $\{\rho_i\}$ forms a family of seminorms on $C^{2,\alpha}_{\rm loc}(\Omega)$.
For $f_1,f_2\in C^1([0,T];C^{2,\alpha}_{\rm loc}(\Omega))$, we define
$$d(f_1,f_2)=\sum^\infty_{i=1}2^{-i}\frac{\rho_i(f_1-f_2)}{1+\rho_i(f_1-f_2)}.$$
Then $d$ is a metric on $C^1([0,T];C^{2,\alpha}_{\rm loc}(\Omega))$.
It is clear that $f_k\rightarrow f$ with respect to $d$ if and only if
$\rho_i(f_k-f)\rightarrow 0$ for all $i$. The topology induced by the metric $d$
is independent of the choice of the sequence $\{\Omega_i\}^\infty_{i=1}$.
A metric on the space $C^{2,\alpha}_{\rm loc}(\Omega)$ can be defined similarly if
we replace the seminorm above by
$$\rho_i(f)=\| f\|_{C^{2,\alpha}(\Omega_i)}$$
for $f\in C^{2,\alpha}_{\rm loc}(\Omega)$ and $i\in\Bbb N$.

\smallskip
Now we fix a sequence $\{\Omega_i\}^\infty_{i=1}$ of compact convex sets in $\Bbb R^2_+$
such that the conditions (i) and (ii) hold. Let $I_i$ denote the projection of $\Omega_i$ to $x_1$-axis.
Then $\{ I_i\}^\infty_{i=1}$ forms a sequence of compact sets in $\Bbb R$ satisfying (i) and (ii).
As stated above, the sequences $\{\Omega_i\}^\infty_{i=1}$ and $\{ I_i\}^\infty_{i=1}$
induce metrics $d_1$ on $C^{2,\alpha}_{\rm loc}(\Bbb R^2_+)$, $d_2$ on $C^1([0,T];C^{2,\alpha}_{\rm loc}(\Bbb R))$,
and $d_3$ on $C^1([0,T];C^{2,\alpha}_{\rm loc}(\Bbb R^2_+))$ respectively.
In Theorem \ref{Classical}, we get that for a given pair
$$(g,h)\in X_{\rm c}:=[H^2(\Bbb R^2_+)\cap C^{2,\alpha}_{\rm loc}(\Bbb R^2_+)]
\times C^1([0,T];H^2(\Bbb R)\cap C^{2,\alpha}_{\rm loc}(\Bbb R))$$
of initial and a boundary data, then \eqref{GBBMB_eq} admits a unique classical solution
$$u\in Y:=C([0,T];H^2(\Bbb R^2_+))\cap C^1([0,T];C^{2,\alpha}_{\rm loc}(\Bbb R^2_+)).$$
If we let $\mathcal{L}_{\rm c}$ denote the mapping that takes the pair $(g,h)$
into the corresponding classical solution $u$, then
\begin{equation}\label{mapping}
\mathcal{L}_{\rm c}:X_{\rm c}\rightarrow Y.
\end{equation}
We remark that if $(M,d_M)$ and $(N,d_N)$ are two metric spaces, then the product space $M\times N$ is
a metric space with metric $d_{M\times N}(x,y)=d_M(x,y)+d_N(x,y)$ for $x,y\in M\times N$,
and their intersection $M\cap N$ is a metric space with metric $d_{M\cap N}(x,y)=d_M(x,y)+d_N(x,y)$ for $x,y\in M\cap N$.

We can immediately conclude that $\mathcal{L}_{\rm c}$ is a continuous mapping in \eqref{mapping}
if we prove that $i\circ\mathcal{L}_{\rm c}$ and $j\circ\mathcal{L}_{\rm c}$ are both continuous
where $i$ and $j$ are the natural inclusions of $Y$ into $C([0,T];H^2(\Bbb R^2_+))$ and
$C^1([0,T];C^{2,\alpha}_{\rm loc}(\Bbb R^2_+))$ respectively.

\begin{lem}\label{s2}
Suppose that $\phi\in C^3(\Bbb R,\Bbb R^2)$ satisfies the conditions \eqref{phi-1}-\eqref{phi-3}.
Then $\mathcal{L}_{\rm c}$ is continuous.
\end{lem}

\begin{proof}
By the discussions before the statement of this lemma, it suffices to prove
\begin{equation}\label{s2-1}
\mathcal{L}_{\rm c}:X_{\rm c}\rightarrow C([0,T];H^2(\Bbb R^2_+))\quad\text{and}\quad\mathcal{L}_{\rm c}:X_{\rm c}\rightarrow C^1([0,T];C^{2,\alpha}_{\rm loc}(\Bbb R^2_+))
\end{equation}
are both continuous. Comparing the metrics of the spaces $X_{\rm m}$ and $X_{\rm c}$, we can easily get the continuity of
the mapping $\mathcal{L}_{\rm c}:X_{\rm c}\rightarrow C([0,T];H^2(\Bbb R^2_+))$ from Lemma \ref{s1}.
In this proof, we focus on showing that the second mapping of \eqref{s2-1} is sequentially continuous.

Let $(g_i,h_i)\in X_{\rm c}$ and $u_i=\mathcal{L}_{\rm c}(g_i,h_i)$ be the classical solution of \eqref{GBBMB_eq}
corresponding to the initial data $g_i$ and the boundary data $h_i$, $i=1,2$.
Let $w$, $\widetilde{g}$, and $h$ be defined as in the proof of Lemma \ref{s1}.
Then $w$ satisfies the initial-boundary value problem \eqref{s1-1}, the integral equation \eqref{s1-2}, and the estimate \eqref{s1-6}.
Let $\Omega$, $\Omega'$, and $\Omega''$ be any given three compact convex sets in
$\Bbb R^2_+$ with $\Omega''\subset\subset\Omega'\subset\subset\Omega$ and let $I$ and $I'$ be the projections
of $\Omega$ and $\Omega'$ to the $x_1$-axis respectively. First, we take $C^{1,\alpha}(\Omega')$ norm
on both sides of \eqref{s1-2} and use Sobolev's inequality to obtain
\begin{align}
&\| w\|_{C^{1,\alpha}(\Omega')} \nonumber\\
&\quad\le \|\widetilde{g}\|_{C^{1,\alpha}(\Omega')} \hspace{-1pt}+\hspace{-1pt}  \left\| (\rI-\Delta)^{-1}\left(\{ h_{x_1 x_1}(x_1,t)-h_{x_1 x_1}(x_1,0)\}e^{-x_2}\right)\right\|_{C^{2,\alpha}(\Omega')} \nonumber\\
&\qquad +\int^t_0\left\|(\rI-\Delta)^{-1}\left\{\div(\phi(v_1 \hspace{-1pt}+\hspace{-1pt} h_1 e^{-x_2}))
-\div(\phi(v_2 \hspace{-1pt}+\hspace{-1pt} h_2 e^{-x_2}))\right\}\right\|_{C^{1,\alpha}(\Bbb R^2_+)}d\tau \nonumber\\
&\quad \le \|\widetilde{g}\|_{C^{1,\alpha}(\Omega')}  \hspace{-1pt}+\hspace{-1pt} \left\| (\rI-\Delta)^{-1}\left(\{ h_{x_1 x_1}(x_1,t)-h_{x_1 x_1}(x_1,0)\}e^{-x_2}\right)\right\|_{C^{2,\alpha}(\Omega')} \label{s2-2} \\
&\qquad +\int^t_0\left\|(\rI-\Delta)^{-1}\left\{\div(\phi(v_1 \hspace{-1pt}+\hspace{-1pt} h_1 e^{-x_2}))
-\div(\phi(v_2 \hspace{-1pt}+\hspace{-1pt} h_2 e^{-x_2}))\right\}\right\|_{H^3(\Bbb R^2_+)}d\tau. \nonumber
\end{align}
Lemma \eqref{reg-2} implies
\begin{equation}\label{s2-3}
\aligned
&\left\| (\rI-\Delta)^{-1}\left(\{ h_{x_1 x_1}(x_1,t)-h_{x_1 x_1}(x_1,0)\}e^{-x_2}\right)\right\|_{C^{2,\alpha}(\Omega')}\\
&\qquad\le C\left\{\left\|\{ h_{x_1 x_1}(x_1,t)-h_{x_1 x_1}(x_1,0)\}e^{-x_2}\right\|_{C^{0,\alpha}(\Omega)}\right.\\
&\qquad\qquad +\left.\left\|\{ h_{x_1 x_1}(x_1,t)-h_{x_1 x_1}(x_1,0)\}e^{-x_2}\right\|_{L^2(\Omega)}\right\}\\
&\qquad\le C\left(\| h\|_{C^{2,\alpha}(I)}+\| h\|_{H^2(I)}\right),
\endaligned
\end{equation}
where $C$ depends only on the distance between $\Omega'$ and $\partial\Omega$.
Employing Proposition \ref{reg-1} and applying mean value theorem together with \eqref{phi-1}
to \eqref{s1-3.5}, we derive that
\begin{equation*}
\aligned
&\left\|(\rI-\Delta)^{-1}\left\{\div(\phi(v_1+h_1 e^{-x_2}))
-\div(\phi(v_2+h_2 e^{-x_2}))\right\}\right\|_{H^3(\Bbb R^2_+)}\\
&\qquad\le C\left\|\div(\phi(v_1+h_1 e^{-x_2}))
-\div(\phi(v_2+h_2 e^{-x_2}))\right\|_{H^1(\Bbb R^2_+)}\\
&\qquad\le C\left\{\| w+he^{-x_2}\|_{W^{1,4}(\Bbb R^2_+)}\|\nabla(v_1+h_1 e^{-x_2})\|_{L^4(\Bbb R^2_+)}\right.\\
&\qquad\qquad +\| w+he^{-x_2}\|_{L^\infty(\Bbb R^2_+)}\|\nabla(v_1+h_1 e^{-x_2})\|_{H^1(\Bbb R^2_+)}\\
&\qquad\qquad +\|v_2+h_2 e^{-x_2}\|_{W^{1,4}(\Bbb R^2_+)}\|\nabla(w+he^{-x_2})\|_{L^4(\Bbb R^2_+)}\\
&\qquad\qquad +\left.(1+\|v_2+h_2 e^{-x_2}\|_{L^\infty(\Bbb R^2_+)})\|\nabla(w+he^{-x_2})\|_{H^1(\Bbb R^2_+)}\right\}.
\endaligned
\end{equation*}
Thus, Sobolev's inequality gives
\begin{equation}\label{s2-4}
\aligned
&\left\|(\rI-\Delta)^{-1}\left\{\div(\phi(v_1+h_1 e^{-x_2}))
-\div(\phi(v_2+h_2 e^{-x_2}))\right\}\right\|_{H^3(\Bbb R^2_+)}\\
&\qquad\le C\left(\| w\|_{H^2(\Bbb R^2_+)}+\| h\|_{H^2(\Bbb R)}\right),
\endaligned
\end{equation}
where $C$ depends only on $v_1$, $v_2$, $h_1$, $h_2$, and $\phi$.
Combining the estimates \eqref{s2-2}-\eqref{s2-4}, we get
\begin{equation}\label{s2-5}
\aligned
\| w\|_{C^{1,\alpha}(\Omega')}&\le \|\widetilde{g}\|_{C^{1,\alpha}(\Omega')}
+C\left(\| h\|_{C^1([0,T];C^{2,\alpha}(I))}+\| h\|_{C([0,T];H^2(I))}\right)\\
&\quad +CT\left(\| w\|_{C([0,T];H^2(\Bbb R^2_+))}+\| h\|_{C([0,T];H^2(\Bbb R))}\right)\\
&\le \|\widetilde{g}\|_{C^{1,\alpha}(\Omega')}+C\| h\|_{C^1([0,T];C^{2,\alpha}(I))}\\
&\quad +C(1+T)\| h\|_{C([0,T];H^2(\Bbb R))}+CT\| w\|_{C([0,T];H^2(\Bbb R^2_+))}.
\endaligned
\end{equation}
Next, by taking $C^{2,\alpha}(\Omega'')$ norm on both sides of \eqref{s1-2}, we have
\begin{align}
& \| w\|_{C^{2,\alpha}(\Omega'')} \le \|\widetilde{g}\|_{C^{2,\alpha}(\Omega'')} \nonumber\\
&\quad + \left\| (\rI-\Delta)^{-1}\left(\{ h_{x_1 x_1}(x_1,t)-h_{x_1 x_1}(x_1,0)\}e^{-x_2}\right)\right\|_{C^{2,\alpha}(\Omega'')} \label{s2-6}\\
&\quad +\int^t_0\left\|(\rI-\Delta)^{-1}\left\{\div(\phi(v_1+h_1 e^{-x_2}))
-\div(\phi(v_2+h_2 e^{-x_2}))\right\}\right\|_{C^{2,\alpha}(\Omega'')}d\tau. \nonumber
\end{align}
We use Proposition \ref{reg-2} to deduce that
\begin{equation}\label{s2-7}
\aligned
&\left\|(\rI-\Delta)^{-1}\left\{\div(\phi(v_1+h_1 e^{-x_2}))
-\div(\phi(v_2+h_2 e^{-x_2}))\right\}\right\|_{C^{2,\alpha}(\Omega'')}\\
&\qquad\le C\left(\left\|\div(\phi(v_1+h_1 e^{-x_2}))
-\div(\phi(v_2+h_2 e^{-x_2}))\right\|_{C^{0,\alpha}(\Omega')}\right.\\
&\qquad\qquad +\left.\left\|\div(\phi(v_1+h_1 e^{-x_2}))
-\div(\phi(v_2+h_2 e^{-x_2}))\right\|_{L^2(\Omega')}\right).
\endaligned
\end{equation}
For the estimate of the last term in the right hand side of \eqref{s2-7},
we use the proof of Lemma \ref{s1} to obtain
\begin{equation}\label{s2-7.5}
\left\|\div(\phi(v_1 \hspace{-1pt}+\hspace{-1pt} h_1 e^{-x_2}))
-\div(\phi(v_2 \hspace{-1pt}+\hspace{-1pt} h_2 e^{-x_2}))\right\|_{L^2(\Omega')} \hspace{-1pt}\le\hspace{-1pt} C\| w \hspace{-1pt}+\hspace{-1pt} he^{-x_2}\|_{H^2(\Bbb R^2_+)}.
\end{equation}
In view of \eqref{s1-3.5} and the convexity of $\Omega'$,
\begin{equation*}
\aligned
&\left\|\div(\phi(v_1+h_1 e^{-x_2}))
-\div(\phi(v_2+h_2 e^{-x_2}))\right\|_{C^{0,\alpha}(\Omega')}\\
&\qquad\le \left\|\phi'(v_1+h_1 e^{-x_2})-\phi'(v_2+h_2 e^{-x_2})\right\|_{C^1(\Omega')}
 \left\|\nabla(v_1+h_1 e^{-x_2})\right\|_{C^{0,\alpha}(\Omega')}\\
&\qquad\quad +\left\|\phi'(v_2+h_2 e^{-x_2})\right\|_{C^1(\Omega')}
 \left\|\nabla(w+he^{-x_2})\right\|_{C^{0,\alpha}(\Omega')}.
\endaligned
\end{equation*}
Since $\phi\in C^3(\Bbb R,\Bbb R^2)$ satisfies the conditions \eqref{phi-1}-\eqref{phi-3},
\begin{equation*}
\left\|\phi'(v_2+h_2 e^{-x_2})\right\|_{C^1(\Omega')}\le C(1+\| v_2+h_2 e^{-x_2}\|_{C^1(\Omega')})\le C
\end{equation*}
and
\begin{equation*}
\aligned
&\left\|\phi'(v_1+h_1 e^{-x_2})-\phi'(v_2+h_2 e^{-x_2})\right\|_{C^1(\Omega')}\\
&\qquad\le \left\|\phi''(v_1+h_1 e^{-x_2})-\phi''(v_2+h_2 e^{-x_2})\right\|_{C(\Omega')}
 \left\|\nabla(v_1+h_1 e^{-x_2})\right\|_{C(\Omega')}\\
&\qquad\quad +\left\|\phi''(v_2+h_2 e^{-x_2})\right\|_{C(\Omega')}\left\|\nabla(w+h e^{-x_2})\right\|_{C(\Omega')}\\
&\qquad\le C\left\| w+h e^{-x_2}\right\|_{C^{1,\alpha}(\Omega')},
\endaligned
\end{equation*}
where $C$ depends only on $v_1$, $v_2$, $h_1$, $h_2$, and $\phi$. Thus, we have
\begin{align}
&\left\|\div(\phi(v_1+h_1 e^{-x_2}))
-\div(\phi(v_2+h_2 e^{-x_2}))\right\|_{C^{0,\alpha}(\Omega')} \le C\left\| w+h e^{-x_2}\right\|_{C^{1,\alpha}(\Omega')} \nonumber\\
&\qquad\le C\|\widetilde{g}\|_{C^{1,\alpha}(\Omega')}+C\| h\|_{C^1([0,T];C^{2,\alpha}(I))} \label{s2-8}\\
&\qquad\quad +C(1+T)\| h\|_{C([0,T];H^2(\Bbb R))}+CT\| w\|_{C([0,T];H^2(\Bbb R^2_+))}, \nonumber
\end{align}
where we used \eqref{s2-5} in the last inequality.
The estimates \eqref{s2-3}, \eqref{s2-6}-\eqref{s2-8} yield
\begin{equation*}
\aligned
&\| w\|_{C^1([0,T];C^{2,\alpha}(\Omega''))} \le C(1+T)\|\widetilde{g}\|_{C^{2,\alpha}(\Omega')}+C(1+T)\| h\|_{C^1([0,T];C^{2,\alpha}(I))}\\
&\qquad\quad +C(1+T)^2\| h\|_{C([0,T];H^2(\Bbb R))}+CT(1+T)\| w\|_{C([0,T];H^2(\Bbb R^2_+))}.
\endaligned
\end{equation*}
It follows from \eqref{s1-6} that
\begin{equation}\label{s2-9}
\aligned
\| w\|_{C^1([0,T];C^{2,\alpha}(\Omega''))}\le e^{CT}&\Big\{\|\widetilde{g}\|_{C^{2,\alpha}(\Omega')}+\|\widetilde{g}\|_{H^2(\Bbb R^2_+)}\\
&\quad +\| h\|_{C^1([0,T];C^{2,\alpha}(I))}+\| h\|_{C([0,T];H^2(\Bbb R))}\Big\}.
\endaligned
\end{equation}

Finally, let $\{(\frak{g}_k,\frak{h}_k)\}^\infty_{k=1}$ be a sequence of $X_{\rm c}$ that converges to $(\frak{g}_0,\frak{h}_0)$ in $X_{\rm c}$.
Suppose that $\frak{u}_k=\mathcal{L}_{\rm c}(\frak{g}_k,\frak{h}_k)$, $k\in\Bbb N\cup\{ 0\}$ be the corresponding classical solutions
of \eqref{GBBMB_eq} with respect to the initial data $\frak{g}_k$ and the boundary data $\frak{h}_k$.
Set, for $k\in\Bbb N\cup\{ 0\}$,
\begin{equation*}
\begin{cases}
\frak{v}_k=\frak{u}_k-\frak{h}_k e^{-x_2},\\
\widetilde{\frak{g}}_k=\frak{g}_k-\frak{h}_k e^{-x_2}.
\end{cases}
\end{equation*}
We define, for $k\in\Bbb N$,
\begin{equation*}
\begin{cases}
\frak{w}_k=\frak{v}_k-\frak{v}_0,\\
\widetilde{\frak{g}}_{k,0}=\widetilde{\frak{g}}_k-\widetilde{\frak{g}}_0,\\
\frak{h}_{k,0}=\frak{h}_k-\frak{h}_0.
\end{cases}
\end{equation*}
Since $\{(\frak{g}_k,\frak{h}_k)\}^\infty_{k=1}$ converges to $(\frak{g}_0,\frak{h}_0)$ in $X_{\rm c}$,
we have
\begin{enumerate}
\item[(a)] $d_1(\widetilde{\frak{g}}_k,\widetilde{\frak{g}}_0)\rightarrow 0$,
\item[(b)] $\|\widetilde{\frak{g}}_k-\widetilde{\frak{g}}_0\|_{H^2(\Bbb R^2_+)}\rightarrow 0$,
\item[(c)] $d_2(\frak{h}_k,\frak{h}_0)\rightarrow 0$,
\item[(d)] $\|\frak{h}_k-\frak{h}_0\|_{C([0,T];H^2(\Bbb R))}\rightarrow 0$;
\end{enumerate}
or equivalently,
\begin{enumerate}
\item[(a$'$)] $\|\widetilde{\frak{g}}_{k,0}\|_{C^{2,\alpha}(\Omega_i)}\rightarrow 0$ for all $i\in\Bbb N$,
\item[(b$'$)] $\|\widetilde{\frak{g}}_{k,0}\|_{H^2(\Bbb R^2_+)}\rightarrow 0$,
\item[(c$'$)] $\|\frak{h}_{k,0}\|_{C^1([0,T];C^{2,\alpha}(I_i))}\rightarrow 0$ for all $i\in\Bbb N$,
\item[(d$'$)] $\|\frak{h}_{k,0}\|_{C([0,T];H^2(\Bbb R))}\rightarrow 0$.
\end{enumerate}
For any fixed $i\in\Bbb N$, applying the estimate \eqref{s2-9}, we get
\begin{equation*}
\aligned
\|\frak{w}_k\|_{C^1([0,T];C^{2,\alpha}(\Omega_i))}\le e^{CT}
&\Big\{\|\widetilde{\frak{g}}_{k,0}\|_{C^{2,\alpha}(\Omega_{i+1})}+\|\widetilde{\frak{g}}_{k,0}\|_{H^2(\Bbb R^2_+)}\\
&\ +\|\frak{h}_{k,0}\|_{C^1([0,T];C^{2,\alpha}(I_{i+1}))}+\|\frak{h}_{k,0}\|_{C([0,T];H^2(\Bbb R))}\Big\}\rightarrow 0,
\endaligned
\end{equation*}
and hence
$$\aligned
&\|\mathcal{L}_{\rm c}(\frak{g}_k,\frak{h}_k)-\mathcal{L}_{\rm c}(\frak{g}_0,\frak{h}_0)\|_{C^1([0,T];C^{2,\alpha}(\Omega_i))}\\
&\qquad\le \|\frak{w}_k\|_{C^1([0,T];C^{2,\alpha}(\Omega_i))}+\|\frak{h}_{k,0}\|_{C^1([0,T];C^{2,\alpha}(\Omega_i))}\\
&\qquad\le e^{CT}\Big\{\|\widetilde{\frak{g}}_{k,0}\|_{C^{2,\alpha}(\Omega_{i+1})}+\|\widetilde{\frak{g}}_{k,0}\|_{H^2(\Bbb R^2_+)}\\
&\qquad\qquad\quad +\|\frak{h}_{k,0}\|_{C^1([0,T];C^{2,\alpha}(I_{i+1}))}+\|\frak{h}_{k,0}\|_{C([0,T];H^2(\Bbb R))}\Big\}\rightarrow 0,
\endaligned$$
which implies that
$$
d_3\left(\mathcal{L}_{\rm c}(\frak{g}_k,\frak{h}_k),\mathcal{L}_{\rm c}(\frak{g}_0,\frak{h}_0)\right)\rightarrow 0
$$
which is equivalent to that the mapping $\mathcal{L}_{\rm c}:X_{\rm c}\rightarrow C^1([0,T];C^{2,\alpha}_{\rm loc}(\Bbb R^2_+))$ is sequentially continuous.
\end{proof}

\vskip 0.5cm
\section{Results for the GBBM-Burgers equation}\label{sec:nu1ne0}
\setcounter{equation}{0}

The purpose of this section is to generalize the above results to the 2D GBBM-Burgers equation, that is, equation \eqref{GBBMB_eq} with $\nu_1=1$. The results established in previous sections also hold for the  GBBM-Burgers equation.

\smallskip
The proofs of Theorems \ref{Classical} and \ref{stability} for the case when $\nu_1 = 1$ are essentially the same as those for the case when $\nu_1=0$. In fact, as in the case when $\nu_1 = 0$, we rewrite equation (\ref{GBBMB_eq}) as
\begin{subequations}\label{GBBMB2}
\begin{alignat}{2}
(\rI - \Delta) v_t + \Delta v + \div \left(\phi(v+he^{-x_2})\right)
&= \widetilde{h} e^{-x_2} \qquad&&\text{in}\quad\bbR^2_+ \times (0,T), \\
v &= \widetilde{g} &&\text{on}\quad\bbR^2_+ \times \{t=0\}, \\
v &= 0 &&\text{on}\quad \bdy\bbR^2_+ \times (0,T),
\end{alignat}
\end{subequations}
where $\widetilde{g}$ is again given by (\ref{defn:gtilde}) and $\widetilde{h}$ is defined by
$$\widetilde{h}(x,t)=h_{x_1 x_1 t}(x,t) - h_{x_1 x_1}(x,t) - h(x,t).$$
(\ref{GBBMB2}a) can be modified as
\begin{equation}\label{modify}
v+v_t = (\rI-\Delta)^{-1}\{ v+\widetilde{h}e^{-x_2}-\div(\phi(v+he^{-x_2}))\}
\end{equation}
which suggest that
\begin{align}
v(x,t) &= e^{-t} \widetilde{g}(x) + \int_0^t e^{-(t-s)} (\rI-\Delta)^{-1}\{v + \widetilde{h}e^{-x_2}-\div(\phi(v+he^{-x_2}))\} ds \nonumber\\
&= e^{-t} \widetilde{g}(x) + \int_0^t e^{-(t-s)} (\rI-\Delta)^{-1} \big\{ (h_{x_1 x_1 s} - h_{x_1 x_1} - h) e^{-x_2}\big\} ds \\
&\quad + \int_0^t e^{-(t-s)} (\rI-\Delta)^{-1} \{v - \div(\phi(v+he^{-x_2})) \} ds \,. \nonumber
\end{align}
Based on the fact that $e^{-(t-s)}$ is bounded by $1$ for $s\in [0,t]$, exactly the same procedure of proving the existence of a unique solution (using the contraction mapping principle) for the case $\nu_1 = 0$ can be applied to yield the results corresponding to Theorems \ref{Classical} and \ref{stability}.

\vskip .3in
{\bf Acknowledgments.} AC was supported by the National Science Council (NSC) of Taiwan under grant 100-2115-M-008-009-MY3, MH was supported by NSC under grant 101-2115-M-008-005, YL was supported by NSC under grant
102-2115-M-008-001, JW was supported by NSF under grant DMS 1209153, and JY was supported by NSC under grant 101-2115-M-126-002.

\vskip .3in

\end{document}